\documentclass{amsart}

\usepackage[latin1]{inputenc}
\usepackage[english]{babel}
\usepackage{indentfirst}
\usepackage{amssymb}
\usepackage{amsthm}
\usepackage[all]{xy}

\newcommand{\impli}{\Rightarrow}
\newcommand{\N}{\mathbb{N}}
\newcommand{\sub}{\subseteq}
\def\epsilon{\varepsilon}

\newtheorem{theo}{Theorem}[section]
\newtheorem{lem}[theo]{Lemma}
\newtheorem{pro}[theo]{Proposition}
\newtheorem{cor}[theo]{Corollary}
\newtheorem{rem}[theo]{Remark}
\newtheorem{exa}[theo]{Example}
\newtheorem{question}[theo]{Question}

\numberwithin{equation}{section}

\title{Dunford-Pettis type properties in $L_1$ of a vector measure}

\author{Jos\'{e} Rodr\'{i}guez}
\address{Dpto. de Matem\'{a}ticas\\E.T.S. de Ingenieros Industriales de Albacete\\
Universidad de Castilla-La Mancha\\ 02071 Albacete\\ Spain} 
\email{jose.rodriguezruiz@uclm.es, joserr@um.es}


\thanks{The research was supported by grants PID2021-122126NB-C32 
(funded by MCIN/AEI/10.13039/501100011033 and ``ERDF A way of making Europe'', EU) and 
21955/PI/22 (funded by {\em Fundaci\'on S\'eneca - ACyT Regi\'{o}n de Murcia}).}

\begin{document}

\begin{abstract}
Let $\nu$ be a countably additive vector measure defined on a $\sigma$-algebra and taking values in a Banach space. 
In this paper we deal with the following three properties for the Banach lattice $L_1(\nu)$ of all $\nu$-integrable real-valued functions: 
the Dunford-Pettis property, the positive Schur property and being lattice-isomorphic to an AL-space. We give new results and we also provide alternative proofs
of some already known ones.
\end{abstract}

\maketitle

\section{Introduction}

Let $X$ be a Banach space with (topological) dual~$X^*$, let $(\Omega,\Sigma)$ be a measurable space and let $\nu:\Sigma\to X$
be a (countably additive) vector measure. A $\Sigma$-measurable function $f:\Omega \to \mathbb{R}$ is called {\em $\nu$-integrable} if 
it is $|x^*\nu|$-integrable for all $x^*\in X^*$ and, for each $A\in \Sigma$, there is $\int_A f \, d\nu\in X$
such that
$$
	x^*\left ( \int_A f \, d\nu \right)=\int_A f\,d(x^*\nu)  	\quad\text{for all $x^*\in X^*$}.
$$
Here $x^*\nu$ is the signed measure obtained as the composition of~$\nu$ with $x^*$ and
$|x^*\nu|$ denotes its variation. By identifying functions which coincide except to a $\nu$-null set
(where $A\in \Sigma$ is said to be {\em $\nu$-null} if $\nu(B)=0$ for every $B\in \Sigma$ with~$B\sub A$), 
the set $L_1(\nu)$ of all (equivalence classes of) $\nu$-integrable functions is a Banach lattice with the $\nu$-a.e. order and the norm
$$
	\|f\|_{L_1(\nu)}:=\sup_{x^*\in B_{X^*}}\int_\Omega |f|\,d|x^*\nu|.
$$
Here $B_{X^*}$ denotes the closed unit ball of~$X^*$. Let us agree to say that $L_1(\nu)$ is the $L_1$ space of the vector measure~$\nu$.

Every Banach lattice with order continuous norm and a weak unit is lattice-isometric to the $L_1$ space of a vector measure, \cite[Theorem~8]{cur1}
(cf. \cite[Proposition~2.4]{dep-alt}). Such a representation is not unique. For instance, the usual space $L_1[0,1]$ is equal to $L_1(\nu_i)$ for each one of the 
following $X_i$-valued vector measures $\nu_i$ defined on the Borel $\sigma$-algebra of~$[0,1]$:
\begin{itemize}
\item $X_1:=\mathbb R$ and $\nu_1(A):=\lambda(A)$ (the Lebesgue measure of~$A$);
\item $X_2:=L_1[0,1]$ and $\nu_2(A):=\chi_A$ (the characteristic function of~$A$);
\item $X_3:=c_0$ and $\nu_3(A):=(\int_A r_n \, d\lambda)_{n\in \N}$, where $(r_n)_{n\in \N}$ is the sequence of Rademacher functions.
\end{itemize}
The structure of the space~$L_1(\nu)$ can be greatly conditioned by certain properties of~$\nu$. For complete information on these spaces and their important
role in Banach lattices and operator theory, we refer the reader to the monograph~\cite{oka-alt} and the papers 
\cite{cal-alt-6,cur-ric-6,cur-ric-5,nyg-rod,oka-rod-san,rod16,rod-rue}.

The inclusion map 
$$
	\iota_\nu: L_1(|\nu|) \to L_1(\nu)   
$$ 
is a well-defined injective lattice-homomorphism, where $|\nu|$ is the variation of~$\nu$ (see, e.g., \cite[Lemma~3.14]{oka-alt}). If $\iota_\nu$ is surjective, then it is a lattice-isomorphism 
and, moreover, we have $|\nu|(\Omega)<\infty$. Curbera~\cite{cur2} addressed the question of when the $L_1$ space
of a vector measure is lattice-isomorphic to an AL-space. Recall that a 
Banach lattice~$E$ is said to be an {\em AL-space} if its norm satisfies $\|x+y\|=\|x\|+\|y\|$ whenever $x,y\in E$ are disjoint, which is equivalent to saying that
$E$ is lattice-isometric to the usual space $L_1(\mu)$ of a non-negative measure~$\mu$ (see, e.g., \cite[Theorem~4.27]{ali-bur}). 
It turns out that $L_1(\nu)$ is lattice-isomorphic to an AL-space if and only if $\iota_\nu$ is surjective, \cite[Proposition~2]{cur2}. This is also equivalent
to the fact that the {\em integration operator} of~$\nu$, that is, the norm~$1$ operator
$$
	I_\nu:L_1(\nu)\to X,\quad I_\nu(f):=\int_\Omega f \, d\nu \quad\text{for all $f\in L_1(\nu)$},
$$
is {\em cone absolutely summing} (i.e., the series $\sum_{n=1}^\infty I_\nu(f_n)$ is absolutely convergent 
whenever $\sum_{n\in \N}f_n$ is unconditionally convergent and $f_n \in L_1(\nu)^+$ for all $n\in \N$), \cite[Proposition~3.1]{cur0}. 
As usual, given a Banach lattice~$E$, we denote by $E^+$ its positive cone, that is, $E^+:=\{x\in E: x\geq 0\}$.
At this point we should stress that if a Banach lattice is isomorphic (just as a Banach space) to an AL-space, then it is lattice-isomorphic to
an AL-space~\cite{abr-woj} (cf. \cite[Proposition~2.1]{hev-alt}). 

An operator between Banach spaces is said to be {\em Dunford-Pettis} (or {\em completely continuous})
if it maps weakly null sequences to norm null ones. The space $L_1(\mu)$ of a non-negative measure~$\mu$ has the {\em Dunford-Pettis property}, that is, every weakly compact operator from~$L_1(\mu)$ to an arbitrary Banach space is Dunford-Pettis (see, e.g., \cite[Theorem~5.4.5]{alb-kal} or \cite[Theorem~5.85]{ali-bur}). 
In general, this is not true for the $L_1$ space of a vector measure. Indeed,
reflexive infinite-dimensional Banach spaces fail the Dunford-Pettis property and, as we have already mentioned, spaces like $\ell_p$ and $L_p[0,1]$ for $1<p<\infty$ 
can be seen as $L_1$ spaces of a vector measure. On the other side, there are $L_1$ spaces of a vector measure having the Dunford-Pettis property
which are not lattice-isomorphic to an AL-space, like~$c_0$. Curbera showed in \cite[Theorem~4]{cur3} that $L_1(\nu)$
has the Dunford-Pettis property if $\nu$ has $\sigma$-finite variation and $X$ has the Schur property (i.e., every weakly null sequence in~$X$ is norm null).
In fact, he proved that: 
\begin{enumerate}
\item[(i)] $L_1(\nu)$ has the positive Schur property whenever $X$ has the Schur property.
\item[(ii)] If $L_1(\nu)$ has the positive Schur property and $\nu$ has $\sigma$-finite variation,
then $L_1(\nu)$ has the Dunford-Pettis property (cf. \cite[Section~3.2]{cal-alt-6}). 
\end{enumerate}
Recall that a Banach lattice $E$ is said to have the {\em positive Schur property} if every weakly null sequence in~$E^+$ is norm null.
Note that statement~(i) can be deduced at once from the fact that $L_1(\nu)$ has the positive Schur property if and only if the integration operator
$I_\nu$ is {\em almost Dunford-Pettis} (i.e., $(I_\nu(f_n))_{n\in \N}$ is norm null for every weakly null sequence $(f_n)_{n\in \N}$ in $L_1(\nu)^+$), 
see \cite[Theorem~5.12]{cal-alt-6}. 

The integration operator is undoubtedly a key point in the theory of $L_1$ spaces of a vector measure. Note that
its properties depend on $\nu$ rather than on the space $L_1(\nu)$ itself. For instance, going back to the example at the beginning, we have:
\begin{itemize}
\item $I_{\nu_1}$ is the functional given by $I_{\nu_1}(f)=\int_{[0,1]} f \, d\lambda$;
\item $I_{\nu_2}$ is the identity operator on $L_1[0,1]$;
\item $I_{\nu_3}: L_1[0,1]\to c_0$ is the operator given by $I_{\nu_3}(f)=(\int_{[0,1]} r_n f \, d\lambda)_{n\in \N}$, which is
strictly singular but fails to be weakly compact.
\end{itemize}
It is known that $L_1(\nu)$ is lattice-isomorphic to an AL-space whenever $I_\nu$ is compact (see \cite[Theorem~1]{oka-alt2}, 
cf. \cite[Theorem~2.2]{oka-alt4} and \cite[Theorem~3.3]{cal-alt-5}), absolutely $p$-summing for $1\leq p <\infty$ (see \cite[Theorem~2.2]{oka-alt3}) 
or, more generally, Dunford-Pettis and Asplund (see \cite[Theorem~3.3]{rod15}). Recall that an operator between Banach spaces is said to be {\em Asplund}
if it factors through a Banach space which is Asplund (i.e., all of its separable subspaces have separable dual).
In particular, $L_1(\nu)$ is lattice-isomorphic to an AL-space if $I_\nu$ is Dunford-Pettis
and $X$ is Asplund,~\cite[Theorem~1.3]{cal-alt-5}. This is a partial answer to the following question posed by
Okada, Ricker and Rodr\'{i}guez-Piazza~\cite{oka-alt3}:

\begin{question}\label{equation:ORR}
Suppose that $I_\nu$ is Dunford-Pettis and that $X$ contains no subspace isomorphic to~$\ell_1$.
Is $L_1(\nu)$ lattice-isomorphic to an AL-space? 
\end{question}

They showed that this is the case if, in addition, $X$ has an unconditional Schauder basis, \cite[Theorem~1.2]{oka-alt3}.
Note that any Banach space with an unconditional Schauder basis and no subspace isomorphic to~$\ell_1$ has separable dual
(see, e.g., \cite[Theorem~3.3.1]{alb-kal}). To the best of our knowledge, Question~\ref{equation:ORR} remains open.

In this paper we deal with $L_1$ spaces of a vector measure with focus on the property of being isomorphic to an AL-space, the 
positive Schur property and the Dunford-Pettis property. Our aim is twofold: we include new results and we also present alternative proofs
of some already known ones which hopefully might led to a better understanding of the theory.  
The structure of the paper is as follows.

In Section~\ref{section:preliminaries} we collect some known preliminary facts on $L_1$ spaces of a vector measure that will be needed later. 

In Section~\ref{section:DPintegrationoperator} we revisit the aforementioned positive answer to Question~\ref{equation:ORR}
for Asplund spaces (Corollary~\ref{cor:Asplund}) and the related result for integration operators which are Dunford-Pettis and Asplund
(Corollary~\ref{cor:Asplundoperator}). 

In Section~\ref{section:DP} we show that the positive Schur property of~$L_1(\nu)$ can be characterized by means of a
Dunford-Pettis type property with respect to the so-called ``vector duality'' induced by the integration operator, that is, the continuous bilinear map
$$
	L_1(\nu)\times L_\infty(\nu) \to X, \qquad
	(f,g) \mapsto I_\nu(fg)=\int_\Omega fg \, d\nu
$$
(Theorem~\ref{theo:PSP}). We also give another proof of the aforementioned result of~\cite{cur3} stating that
$L_1(\nu)$ has the Dunford-Pettis property if it has the positive Schur property and $\nu$ has $\sigma$-finite variation
(Corollary~\ref{cor:PSP-DP}). It seems to be an open question whether the assumption on the variation can be dropped, namely: 

\begin{question}\label{equation:ORRdos}
Suppose that $L_1(\nu)$ has the positive Schur property. Does $L_1(\nu)$ have the Dunford-Pettis property? 
\end{question}

Finally, in Example~\ref{exa:Lipecki} we discuss a class of vector measures~$\nu$  
such that $L_1(\nu)$ has the positive Schur property and the Dunford-Pettis property, but fails to be lattice-isomorphic to an AL-space,
among other interesting properties.

\section{Preliminaries}\label{section:preliminaries}

All Banach spaces considered in this paper are real. An {\em operator} is a continuous linear map between Banach spaces.
Given an operator~$T$, its adjoint is denoted by~$T^*$. 
By a {\em subspace} of a Banach space we mean a norm closed linear subspace. 
Let $Z$ be a Banach space. The norm of~$Z$ is denoted by $\|\cdot\|_Z$, or simply $\|\cdot\|$, and we write $B_Z:=\{z\in Z: \|z\|\leq 1\}$ (the closed unit ball of~$Z$). 
The evaluation of $z^*\in Z^*$ at $z\in Z$ is denoted by either
$z^*(z)$ or $\langle z^*,z\rangle$. By a {\em projection} from~$Z$ onto a subspace~$Y \sub Z$
we mean an operator $P:Z\to Z$ such that $P(Z)=Y$ and $P$ is the identity when restricted to~$Y$.
The subspace of~$Z$ generated by a set $H \sub Z$ is denoted by $\overline{{\rm span}}(H)$.

In this section we gather, for the reader's convenience, some known facts on $L_1$ spaces of a vector measure. 
A basic reference on this topic is \cite[Chapter~3]{oka-alt}. 

Throughout this section $X$ is a Banach space, $(\Omega,\Sigma)$ is a measurable space and $\nu\in {\rm ca}(\Sigma,X)$.
As usual, we denote by ${\rm ca}(\Sigma,X)$ the set of all countably
additive $X$-valued vector measures defined on~$\Sigma$.
The {\em range} of~$\nu$ is the set
$$
	\mathcal{R}(\nu):=\{\nu(A): \, A\in \Sigma\}\sub X.
$$ 
The variation and semivariation of~$\nu$ are denoted by~$|\nu|$ and $\|\nu\|$, respectively.
The family of all $\nu$-null sets is denoted by $\mathcal{N}(\nu)$. By a {\em Rybakov control measure} of~$\nu$
we mean a finite non-negative measure of the form $\mu=|x^*\nu|$ for some $x^*\in X^*$ such that 
$\mathcal{N}(\mu) = \mathcal{N}(\nu)$ (see, e.g., \cite[p.~268, Theorem~2]{die-uhl-J}).
Throughout this section $\mu$ is a fixed Rybakov control measure of~$\nu$.

\subsection{$L_\infty$ of a vector measure}
A function $f:\Omega \to \mathbb{R}$ is called {\em $\Sigma$-simple} if it is a linear combination of
functions of the form $\chi_A$, where $A\in \Sigma$. Clearly, all $\Sigma$-simple functions are $\nu$-integrable.
The set of all $\Sigma$-simple functions is norm dense in~$L_1(\nu)$ (see, e.g., \cite[Theorem~3.7(ii)]{oka-alt}), so one has
\begin{equation}\label{eqn:range}
	\overline{I_\nu(L_1(\nu))}
	=
	\overline{{\rm span}}(\mathcal{R}(\nu)).
\end{equation}
More generally, every $\nu$-essentially bounded $\Sigma$-measurable function $f:\Omega \to \mathbb{R}$ is $\nu$-integrable. 
By identifying functions which coincide $\nu$-a.e., the set $L_\infty(\nu)$ of all (equivalence classes of)
$\nu$-essentially bounded $\Sigma$-measurable functions is a Banach lattice with the $\nu$-a.e. order and the $\nu$-essential supremum norm~$\|\cdot\|_{L_\infty(\nu)}$. Of course,
$L_\infty(\nu)$ is equal to the usual spaces $L_\infty(|\nu|)$ and $L_\infty(\mu)$. The inclusion map
$$
	j_\nu: L_\infty(\nu) \to L_1(\nu)
$$
is an injective operator. Moreover, it is weakly compact. Indeed, $j_\nu(B_{L_\infty(\nu)})$ coincides with the order interval $[-\chi_\Omega,\chi_\Omega]$ in $L_1(\nu)$, so
it is weakly compact as $L_1(\nu)$ has order continuous norm (see, e.g., \cite[Theorem~4.9]{ali-bur}). 
Hence, $I_\nu(j_\nu(B_{L_\infty(\nu)}))$ is weakly compact in~$X$. We have the following characterization of relative norm compactness
of~$\mathcal{R}(\nu)$ (see, e.g., \cite[Proposition~2.41]{oka-alt}):

\begin{pro}\label{pro:compactrange}
The following statements are equivalent:
\begin{enumerate}
\item[(i)] $\mathcal{R}(\nu)$ is relatively norm compact.
\item[(ii)] $I_\nu(j_\nu(B_{L_\infty(\nu)}))$ is norm compact.
\end{enumerate}
\end{pro}

\subsection{Composition of a vector measure with an operator}

We will use several times the following fact (see, e.g., \cite[Lemma~3.27]{oka-alt}):

\begin{pro}\label{pro:composition}
Let $T:X\to Y$ be an operator between Banach spaces. Then:
\begin{enumerate}
\item[(i)] The composition $\tilde{\nu}:=T\circ \nu:\Sigma \to Y$
is a countably additive vector measure.
\item[(ii)] Every $\nu$-integrable function is $\tilde{\nu}$-integrable.
\item[(iii)] The inclusion map $u: L_1(\nu) \to L_1(\tilde{\nu})$ is an operator and $I_{\tilde{\nu}}\circ u=T\circ I_\nu$.
\end{enumerate} 
\end{pro}

\subsection{L-weakly compact sets and the positive Schur property}
Let $E$ be a Banach lattice. Given a set $W \sub E$, we denote by ${\rm Sol}(W)$ its {\em solid hull}, that is,
the set of all $x\in E$ such that $|x|\leq |y|$ for some $y\in W$. It is known
that if $W$ is relatively weakly compact, then every disjoint sequence in ${\rm Sol}(W)$ is weakly null
(see, e.g., \cite[Theorem~4.34]{ali-bur}). The set $W$ is said to be {\em L-weakly compact}
if it is bounded and every disjoint sequence in ${\rm Sol}(W)$ is norm null. Every 
L-weakly compact set is relatively weakly compact (see, e.g., \cite[Theorem~5.55]{ali-bur}), but the converse does not hold in general.
The following result is well-known (see \cite[Corollaries~2.3.5 and~3.6.8]{mey2}, \cite[Theorem~1.16]{san-hen} and \cite[Lemma~3]{wnu2}):

\begin{pro}\label{pro:PSPabstract}
Let $E$ be a Banach lattice. The following statements are equivalent:
\begin{enumerate}
\item[(i)] $E$ has the positive Schur property.
\item[(ii)] Every disjoint weakly null sequence in~$E$ is norm null.
\item[(iii)] Every disjoint weakly null sequence in~$E^+$ is norm null.
\item[(iv)] Every relatively weakly compact subset of~$E$ is L-weakly compact. 
\end{enumerate}
\end{pro}

Proposition~\ref{pro:Lweaklycompact} below characterizes L-weakly compact sets in the $L_1$ space of a vector measure. We first need to introduce
some terminology. Given $f\in L_1(\nu)$, the map 
$\nu_f:\Sigma \to X$ defined by
$$
	\quad \nu_f(A):=I_\nu(f\chi_A)=\int_A f \, d\nu
	\quad\text{for all $A\in \Sigma$} 
$$
is a countably additive vector measure by the Orlicz-Pettis theorem (see, e.g., \cite[p.~22, Corollary~4]{die-uhl-J}). Note that $\|\nu_f\|(A)=\|f\chi_A\|_{L_1(\nu)}$
for all $A\in \Sigma$. Moreover, $\nu_f$ is {\em $\mu$-continuous}, that is,
for every $\epsilon>0$ there is $\delta>0$ such that 
$\|\nu_f(A)\|\leq \epsilon$ 
for every $A\in \Sigma$ with $\mu(A)\leq \delta$
(see, e.g., \cite[p.~10, Theorem~1]{die-uhl-J}). 
A set $F \sub L_1(\nu)$ is said to be {\em equi-integrable} if the set $\{\nu_f:f\in F\}$ is
{\em uniformly $\mu$-continuous}, that is, for every $\epsilon>0$
there is $\delta>0$ such that 
$$
	\sup_{f\in F}\|\nu_f(A)\| \leq \epsilon
	\quad\text{for every $A\in \Sigma$ with $\mu(A)\leq \delta$}.
$$
The following result can be found in \cite[Proposition~3.6.2]{mey2} and \cite[Lemma~2.37]{oka-alt}
within a more general framework. 

\begin{pro}\label{pro:Lweaklycompact}
Let $F \sub L_1(\nu)$ be a set. The following statements are equivalent:
\begin{enumerate}
\item[(i)] $F$ is L-weakly compact.
\item[(ii)] $F$ is bounded and equi-integrable. 
\item[(iii)] $F$ is {\em approximately order bounded}, i.e., for every $\epsilon>0$ there is $\rho>0$ such that 
$F \sub j_\nu(\rho B_{L_\infty(\nu)})+\epsilon B_{L_1(\nu)}$. 
\end{enumerate}
\end{pro}

We refer the reader to the recent works \cite{bot-alt,che,les-alt} for further results related to the positive Schur property in Banach lattices.

\subsection{Characterization of Dunford-Pettis integration operators}
Proposition~\ref{pro:charDP} below was first proved in \cite[Theorem~5.8]{cal-alt-6}.
We include here a more direct proof for the reader's convenience. One part follows the argument used in \cite[Theorem~4]{cur3} to show that $L_1(\nu)$
has the positive Schur property if $X$ has the Schur property. The following auxiliary lemma will also be used later.

\begin{lem}\label{lem:equi}
Let $(f_n)_{n\in \N}$ be a sequence in~$L_1(\nu)$ such that the sequence $(\nu_{f_n}(A))_{n\in \N}$ is norm convergent for every $A\in \Sigma$. 
Then $(f_n)_{n\in \N}$ is equi-integrable.
\end{lem}
\begin{proof} This follows from the Vitali-Hahn-Saks theorem
(see, e.g., \cite[p.~24, Corollary~10]{die-uhl-J}) applied to the sequence of $\mu$-continuous
vector measures $(\nu_{f_n})_{n\in \N}$.
\end{proof}

\begin{pro}\label{pro:charDP}
The following statements are equivalent:
\begin{enumerate}
\item[(i)] $L_1(\nu)$ has the positive Schur property and $\mathcal{R}(\nu)$ is relatively norm compact.
\item[(ii)] $I_\nu$ is Dunford-Pettis.
\end{enumerate}
\end{pro}
\begin{proof}
(i)$\impli$(ii): Let $F \sub L_1(\nu)$ be a relatively weakly compact set. We will show that $I_\nu(F)$ is relatively norm compact by checking that
for each $\epsilon>0$ there is a norm compact set $K_\epsilon \sub X$ such that $I_\nu(F) \sub K_\epsilon+\epsilon B_X$. 
Fix $\epsilon>0$. Since $F$ is approximately order bounded (by Propositions~\ref{pro:PSPabstract} and~\ref{pro:Lweaklycompact}), there is
$\rho>0$ such that $F \sub j_\nu(\rho B_{L_\infty(\nu)})+\epsilon B_{L_1(\nu)}$. Therefore,
$I_\nu(F) \sub K_\epsilon+\epsilon B_{X}$,
where $K_\epsilon:= I_\nu(j_\nu(\rho B_{L_\infty(\nu)}))$ is norm compact by Proposition~\ref{pro:compactrange}.

(ii)$\impli$(i): Since $j_\nu(B_{L_\infty(\nu)})$ is weakly compact in~$L_1(\nu)$ (see the paragraph preceding Proposition~\ref{pro:compactrange})
and $I_\nu$ is Dunford-Pettis, 
the set $I_\nu(j_\nu(B_{L_\infty(\nu)}))$ is norm compact and so $\mathcal{R}(\nu)$ is relatively norm compact (Proposition~\ref{pro:compactrange}). 
To prove that $L_1(\nu)$ has the positive Schur property it suffices to check that every weakly convergent sequence
is equi-integrable (Propositions~\ref{pro:PSPabstract} and~\ref{pro:Lweaklycompact}). Let $(f_n)_{n\in \N}$ be a sequence in~$L_1(\nu)$ which converges 
weakly to some~$f\in L_1(\nu)$. Then for each $A\in \Sigma$ the sequence $(f_n\chi_A)_{n\in \N}$ converges
weakly to~$f\chi_A$ in~$L_1(\nu)$ (bear in mind that the map $h \mapsto h\chi_A$ is an operator on~$L_1(\nu)$). Since $I_\nu$ is Dunford-Pettis,
for each $A\in \Sigma$ the sequence $(I_{\nu}(f_n\chi_A))_{n\in \N}=(\nu_{f_n}(A))_{n\in \N}$ converges in norm to~$I_{\nu}(f\chi_A)=\nu_f(A)$.
Now, Lemma~\ref{lem:equi} applies to conclude that $(f_n)_{n\in \N}$ is equi-integrable.
\end{proof}

\subsection{The ``vector duality'' induced by the integration operator}\label{subsection:vectorduality}

The following result (see, e.g., \cite[Proposition~3.31]{oka-alt}) shows, in particular, that we have a continuous bilinear map
$$
	L_1(\nu)\times L_\infty(\nu)\to X
$$
defined by
$$
	\quad (f,g) \mapsto I_\nu(fg)=\int_\Omega fg \, d\nu.
$$

\begin{pro}\label{pro:norming}
Let $f\in L_1(\nu)$. Then:
\begin{enumerate}
\item[(i)] For every $g\in L_\infty(\nu)$ the product $fg\in L_1(\nu)$ and 
$$
	\|fg\|_{L_1(\nu)}\leq \|f\|_{L_1(\nu)}\|g\|_{L_\infty(\nu)}.
$$
\item[(ii)] The norm of~$f$ in $L_1(\nu)$ is 
$$
	\|f\|_{L_1(\nu)}=\sup_{g\in B_{L_\infty(\nu)}}\|I_\nu(fg)\|_X.
$$ 
\end{enumerate}
\end{pro}

There are some elements of $L_1(\nu)^*$ which admit a simple description and are helpful for dealing
with the weak topology of~$L_1(\nu)$. By Proposition~\ref{pro:norming}, for each $(g,x^*)\in B_{L_\infty(\nu)}\times B_{X^*}$ we can
define a functional $\gamma_{(g,x^*)} \in B_{L_1(\nu)^*}$ by the formula
$$
	\gamma_{(g,x^*)}(f):=x^*(I_\nu(fg))=\int_\Omega fg \, d(x^*\nu)
	\quad\text{for all $f\in L_1(\nu)$},
$$
and the set
$$
	\Gamma_\nu:=\{\gamma_{(g,x^*)}: \, (g,x^*)\in B_{L_\infty(\nu)}\times B_{X^*}\}
	\sub B_{L_1(\nu)^*}
$$
is norming for~$L_1(\nu)$, that is,
\begin{equation}\label{eqn:Gamma}
	\|f\|_{L_1(\nu)}=\sup_{\gamma\in \Gamma_\nu} \gamma(f)
	\quad\text{for all $f\in L_1(\nu)$}.
\end{equation}
Let $\sigma(L_1(\nu),\Gamma_\nu)$ be the (locally convex Hausdorff) topology on~$L_1(\nu)$ of pointwise convergence
on~$\Gamma_\nu$, which is weaker than the weak topology. 
Proposition~\ref{pro:convergenceproperty} below was first proved in~\cite[Proposition~17]{oka}. It was pointed out
in \cite[Section~4.7]{man-J} that it can also be seen as a corollary of the Rainwater-Simons theorem (see, e.g., \cite[Theorem~3.134]{fab-ultimo})
and the fact that $\Gamma_\nu$ is a James boundary for~$L_1(\nu)$ (i.e., the supremum in~\eqref{eqn:Gamma} is a maximum)
whenever $\mathcal{R}(\nu)$ is relatively norm compact. We refer the reader to \cite[Section~4]{cal-alt-6}
and \cite[Section~2]{cal-alt-7} for more information on this topic.

\begin{pro}\label{pro:convergenceproperty}
Suppose that $\mathcal{R}(\nu)$ is relatively norm compact. Then every bounded and $\sigma(L_1(\nu),\Gamma_\nu)$-convergent sequence in~$L_1(\nu)$ 
is weakly convergent.
\end{pro}

\section{Dunford-Pettis integration operators}\label{section:DPintegrationoperator}

Let $\mathcal{A}$ be an operator ideal. Following~\cite{oka-alt3}, a Banach space~$X$ is said to be 
{\em $\mathcal{A}$-variation admissible} if for every measurable space $(\Omega,\Sigma)$ and for every $\nu \in {\rm ca}(\Sigma,X)$
such that $I_\nu \in \mathcal{A}$, we have $|\nu|(\Omega)<\infty$. The interest of this concept is based on
the following result, \cite[Proposition~1.1]{oka-alt3}, which provides a tool for proving that $L_1$ is lattice-isomorphic to an AL-space
under some additional assumptions.

\begin{pro}\label{pro:OIP}
Let $\mathcal{A}$ be an operator ideal and let $X$ be a Banach space. If $X$ is $\mathcal{A}$-variation admissible, then 
for every measurable space $(\Omega,\Sigma)$ and for every $\nu \in {\rm ca}(\Sigma,X)$ such that $I_\nu \in \mathcal{A}$, the inclusion map $\iota_\nu: L_1(|\nu|) \to L_1(\nu)$
is a lattice-isomorphism.
\end{pro}

The next proposition is elementary.

\begin{pro}\label{pro:stability1}
Let $\mathcal{A}$ be an operator ideal and let $X$ and $Y$ be Banach spaces.
\begin{enumerate}
\item[(i)] If $X$ and $Y$ are isomorphic, then $X$ is $\mathcal{A}$-variation admissible if and only if $Y$ is $\mathcal{A}$-variation admissible.
\item[(ii)] If $X$ is $\mathcal{A}$-variation admissible, then every subspace of~$X$ is $\mathcal{A}$-variation admissible.
\end{enumerate}
\end{pro}
\begin{proof} Let $(\Omega,\Sigma)$ be a measurable space.

(i) If $T:X\to Y$ is an isomorphism and $\nu\in {\rm ca}(\Sigma,X)$, then we can apply Proposition~\ref{pro:composition}
to deduce that a function is $\nu$-integrable if and only if it is $\tilde{\nu}$-integrable, where we write $\tilde{\nu}:=T\circ\nu \in{\rm ca}(\Sigma,Y)$,
and that the identity map $u: L_1(\nu)\to L_1(\tilde{\nu})$ is an isomorphism satisfying
$I_{\tilde{\nu}}\circ u=T\circ I_\nu$. Hence, $I_\nu\in \mathcal{A}$ if and only if $I_{\tilde{\nu}}\in \mathcal{A}$.
Moreover, we have $\|T^{-1}\|^{-1}\cdot |\nu|(\Omega) \leq |\tilde{\nu}|(\Omega)\leq \|T\|\cdot |\nu|(\Omega)$.

(ii) Let $Z\sub X$ be a subspace and let $i:Z\to X$ be the inclusion operator. Fix $\nu\in {\rm ca}(\Sigma,Z)$ such that $I_\nu\in \mathcal{A}$ and define
$\tilde{\nu}:=i\circ\nu \in{\rm ca}(\Sigma,X)$. Then a function is $\nu$-integrable if and only if it is $\tilde{\nu}$-integrable,
the identity map $u: L_1(\nu)\to L_1(\tilde{\nu})$ is an isometry and $I_{\tilde{\nu}}=i\circ I_\nu \circ u^{-1} \in \mathcal{A}$ 
(Proposition~\ref{pro:composition}). Since $X$ is $\mathcal{A}$-variation admissible, we have $|\nu|(\Omega)=|\tilde{\nu}|(\Omega)<\infty$.
\end{proof}

In this section we focus on the operator ideal $\mathcal{A}_{cc}$ of all Dunford-Pettis operators. Our main goal is to provide a somehow simpler proof
of the fact that Asplund spaces are $\mathcal{A}_{cc}$-variation admissible, \cite[Theorem~1.3]{cal-alt-5}, see Corollary~\ref{cor:Asplund} below. 

Part~(i) of the following result was already pointed out in \cite{oka-alt3}:
 
\begin{pro}\label{pro:stability2}
Let $X$ be a Banach space.
\begin{enumerate}
\item[(i)] If $X$ is $\mathcal{A}_{cc}$-variation admissible, then $X$ contains no subspace isomorphic to~$\ell_1$.
\item[(ii)] If every separable subspace of~$X$ is $\mathcal{A}_{cc}$-variation admissible, then $X$ is $\mathcal{A}_{cc}$-variation admissible.
\end{enumerate}
\end{pro}
\begin{proof} (i) By Proposition~\ref{pro:stability1}, its suffices to check that $\ell_1$ is not 
$\mathcal{A}_{cc}$-variation admissible. Since $\ell_1$ is infinite-dimensional, the Dvoretzky-Rogers theorem (see, e.g., \cite[Theorem~1.2]{die-alt})
ensures the existence of an unconditionally convergent series $\sum_{n=1}^\infty x_n$ in~$\ell_1$ which is not absolutely convergent. Now, define $\nu:\mathcal{P}(\N)\to \ell_1$
by $\nu(A):=\sum_{n\in A}x_n$ for all $A \in \mathcal{P}(\N)$ (the power set of~$\N$). Then $\nu\in {\rm ca}(\mathcal{P}(\N),\ell_1)$ 
satisfies $|\nu|(\mathcal{P}(\N))=\sum_{n=1}^\infty \|x_n\|=\infty$, 
while $I_\nu$ is Dunford-Pettis by the Schur property of~$\ell_1$ (see, e.g., \cite[Theorem~2.3.6]{alb-kal}). Hence, $\ell_1$ is not 
$\mathcal{A}_{cc}$-variation admissible.
 
(ii) Let $(\Omega,\Sigma)$ be a measurable space and let $\nu \in {\rm ca}(\Sigma,X)$
such that $I_\nu$ is Dunford-Pettis. Then $\mathcal{R}(\nu)$ is relatively norm compact (Proposition~\ref{pro:charDP}), so
it is separable. Hence, the subspace $Z:=\overline{{\rm span}}(\mathcal{R}(\nu))=\overline{I_\nu(L_1(\nu))}$ 
(see \eqref{eqn:range} at page~\pageref{eqn:range}) is separable. Then $Z$ is
$\mathcal{A}_{cc}$-variation admissible by assumption. Since $\nu$ takes values in~$Z$ 
and $I_\nu$ is Dunford-Pettis, we deduce that $|\nu|(\Omega)<\infty$.
\end{proof}

\subsection{Schauder decompositions and the variation of a vector measure}\label{subsection:Schauder}

Let $X$ be a Banach space. A {\em Schauder decomposition} of~$X$ is a sequence $(X_n)_{n\in \N}$
of (non-zero) subspaces of~$X$ such that each $x\in X$ can be written in a unique way as a convergent series of the form $x=\sum_{n=1}^\infty x_n$, where
$x_n\in X_n$ for all $n\in \N$. In this case, for each $n\in \N$ there is a projection $S_n$ from~$X$ onto~$X_n$
such that $x=\sum_{n=1}^\infty S_n(x)$ for all $x\in X$. For each $k\in \N$, the operator
$P_k:=\sum_{n=1}^k S_n$ is a projection from~$X$ onto the subspace~$\bigoplus_{n=1}^k X_n$,
we have $\sup_{k\in \N}\|P_k\|<\infty$ and the formula $|||x|||=\sup_{k\in \N}\|P_k(x)\|$
defines an equivalent norm on~$X$. Note that $P_{k'}\circ P_{k}=P_{k}\circ P_{k'}=P_{\min\{k,k'\}}$ for all $k,k'\in \N$
and $\|P_k(x)-x\|\to 0$ as $k\to \infty$ for every $x\in X$.
Of course, if $X$ has a Schauder basis $(e_n)_{n\in \N}$, then 
the sequence of $1$-dimensional subspaces generated by each $e_n$ is a Schauder decomposition of~$X$.

The following lemma will be a key tool for proving Theorem~\ref{theo:general} below.

\begin{lem}\label{lem:projection2}
Let $X$ be a Banach space, let $(\Omega,\Sigma)$ be a measurable space and let $\nu \in {\rm ca}(\Sigma,X)$. Suppose that:
\begin{itemize}
\item $|\nu|(\Omega)=\infty$ and $\mathcal{R}(\nu)$ is relatively norm compact.
\item $X$ has a Schauder decomposition $(X_n)_{n\in \N}$ such that $|P_k\circ \nu|(\Omega)<\infty$ for all $k\in \N$, where $P_k$ is the associated projection
from~$X$ onto $\bigoplus_{n=1}^k X_n$. 
\end{itemize}
Then there exist a sequence $(B_j)_{j\in \N}$ of pairwise disjoint elements of $\Sigma\setminus \mathcal{N}(\nu)$,
a strictly increasing sequence $(k_j)_{j\in \N}$ in~$\N$ and $\epsilon>0$ in such a way that the functions $f_j:=\frac{1}{\|\nu\|(B_j)}\chi_{B_j} \in L_1(\nu)$ and the projections
$R_j:=P_{k_{j+1}}-P_{k_{j}}$ satisfy:
\begin{enumerate}
\item[(i)] $\|I_\nu(f_jg)-R_{j}(I_\nu(f_jg))\| \leq 2^{-j}$ for all $g\in B_{L_\infty(\nu)}$ and $j\in \N$.
\item[(ii)] There is $j_0\in \N$ such that $\|I_\nu(f_j)-I_\nu(f_{j'})\|\geq \epsilon$ for all distinct $j,j' \geq j_0$.
\end{enumerate}
\end{lem}
\begin{proof} Write $Q_k:={\rm id}_X-P_k$ for all $k\in \N$ (where ${\rm id}_X$ stands for the identity operator on~$X$). 
Since $\sup_{k\in \N}\|Q_k\|<\infty$ and $\|Q_k(x)\|\to 0$ as $k\to \infty$ for every $x\in X$, the sequence
$(Q_k)_{k\in \N}$ converges to~$0$ uniformly on each norm compact subset of~$X$.
By renorming, we can assume without loss of generality that $\|P_k\|=1$ for all $k\in \N$.

Since $|\nu|(\Omega)=\infty$, there is a sequence $(C_l)_{l\in \N}$ of pairwise disjoint elements of $\Sigma\setminus \mathcal{N}(\nu)$ 
such that $\sum_{l=1}^\infty\|\nu(C_l)\|=\infty$, \cite[Corollary~2]{nyg-pol2}. Fix $\rho>2$ and, for each $l\in \N$, take 
$A_l \in \Sigma \setminus \mathcal{N}(\nu)$ such that $A_l \sub C_l$ and $\rho \|\nu(A_l)\| \geq \|\nu\|(C_l)$.
Then $\sum_{l=1}^\infty \|\nu(A_l)\|=\infty$ and 
\begin{equation}\label{eqn:seminormalized}
	\|\nu(A_l)\| \geq \rho^{-1}\|\nu\|(A_l) \quad\text{for all $l\in \N$}.
\end{equation}

{\em Claim.} There exist two strictly increasing sequences $(k_j)_{j\in \N}$ 
and $(l_j)_{j\in \N}$ in~$\N$ such that for every $j\in \N$ we have:
\begin{enumerate}
\item[($\alpha$)] $\|P_{k_{j}}(I_\nu(\chi_{A_{l_{j+1}}}g))\|\leq 2^{-j-1}\|\nu\|(A_{l_{j+1}})$ for all $g\in B_{L_\infty(\nu)}$;
\item[($\beta$)] $\|Q_{k_{j}}(I_\nu(\chi_{A_{l_j}}g))\|\leq 2^{-j} \|\nu\|(A_{l_j})$ for all $g\in B_{L_\infty(\nu)}$.
\end{enumerate}
Indeed, we proceed by induction. Set $l_1:=1$ and consider 
$$
	K_1:=\big\{I_\nu(\chi_{A_{1}}g): \, g \in B_{L_\infty(\nu)}\big\} \sub X.
$$ 
Since $\mathcal{R}(\nu)$ is relatively norm compact, so is $K_1$ (Proposition~\ref{pro:compactrange}) and therefore we can pick $k_1\in \N$ such that 
$$
	\sup_{x \in K_{1}}\|Q_{k_1}(x)\|\leq \frac{\|\nu\|(A_{1})}{2}.
$$
Hence, ($\beta$) holds for $j=1$.
Suppose now that $k_N,l_N\in \N$ are already chosen for some $N\in \N$. Since $\tilde{\nu}:=P_{k_N} \circ \nu$ satisfies $|\tilde{\nu}|(\Omega)<\infty$
and $\sum_{l=1}^\infty \|\nu\|(A_l)=\infty$, 
there is $l_{N+1}\in \N$ with $l_{N+1}>l_N$ such that
\begin{equation}\label{eqn:variationcontrol}
	  |\tilde{\nu}|(A_{l_{N+1}}) \leq 2^{-N-1}\|\nu\|(A_{l_{N+1}}).
\end{equation}
Observe that for each $g\in B_{L_\infty(\nu)}$ we have
\begin{eqnarray*}
	\|P_{k_N}(I_\nu(\chi_{A_{l_{N+1}}}g))\|& \stackrel{(*)}{=} & \|I_{\tilde{\nu}}(\chi_{A_{l_{N+1}}}g)\| \leq  \|\chi_{A_{l_{N+1}}}g\|_{L_1(\tilde{\nu})} 
	\stackrel{(**)}{\leq} \|\chi_{A_{l_{N+1}}}\|_{L_1(\tilde{\nu})} \\
	& = &  \|\tilde{\nu}\|(A_{l_{N+1}}) \leq |\tilde{\nu}|(A_{l_{N+1}}) \stackrel{\eqref{eqn:variationcontrol}}{\leq} 2^{-N-1}\|\nu\|(A_{l_{N+1}}),
\end{eqnarray*}
where $(*)$ and $(**)$ follow from Propositions~\ref{pro:composition} and~\ref{pro:norming}(i), respectively. 
Hence, ($\alpha$) holds for $j=N$. 
Now, we consider the relatively norm compact subset of~$X$ defined by
$$
	K_{N+1}:=\big\{I_\nu(\chi_{A_{l_{N+1}}}g): \, g \in B_{L_\infty(\nu)}\big\}
$$ 
(apply Proposition~\ref{pro:compactrange} again) and we choose $k_{N+1}\in \N$ with $k_{N+1}>k_N$ such that 
$$
	\sup_{x \in K_{N+1}}\|Q_{k_{N+1}}(x)\| \leq 2^{-N-1}\|\nu\|(A_{l_{N+1}}) .
$$
Therefore, ($\beta$) holds for $j=N+1$.
This finishes the proof of the {\em Claim}.

For each $j\in \N$, define $B_j:=A_{l_{j+1}}$ and let $f_j$ and $R_j$ be as in the statement. 
To check property~(i), take $j\in\N$ and $g\in B_{L_\infty(\nu)}$. Then ($\alpha$) and ($\beta$) imply
\begin{eqnarray*}
	\|I_\nu(f_jg)-R_{j}(I_\nu(f_jg))\|&=&\|Q_{k_{j+1}}(I_\nu(f_jg)) + P_{k_{j}}(I_\nu(f_j g))\| \\
	&\leq & \|Q_{k_{j+1}}(I_\nu(f_jg))\| + \|P_{k_{j}}(I_\nu(f_j g))\| \\ &\leq& \frac{1}{2^{j+1}}+\frac{1}{2^{j+1}}=\frac{1}{2^j}.
\end{eqnarray*}
Finally, we will check that (ii) holds for an arbitrary $0<\epsilon<\rho^{-1}$. 
Choose $j_0\in \N$ large enough such that $\rho^{-1}-2^{-j_0}\geq\epsilon$.  
Take $j'>j\geq j_0$ in~$\N$. Then
\begin{eqnarray*}
	\|I_\nu(f_j)-I_\nu(f_{j'})\| & \geq & \| P_{k_{j+1}}(I_\nu(f_j)-I_\nu(f_{j'}))\| \\ & &  \text{(because $\|P_{k_{j+1}}\|=1$)} \\
	& = & \|I_\nu(f_j)-Q_{k_{j+1}}(I_\nu(f_j))-P_{k_{j+1}}(I_\nu(f_{j'}))\| \\
	& \geq &  \|I_\nu(f_j)\|-\|Q_{k_{j+1}}(I_\nu(f_j))\|-\|P_{k_{j+1}}(I_\nu(f_{j'}))\| \\ 
	& = & \|I_\nu(f_j)\|-\|Q_{k_{j+1}}(I_\nu(f_j))\|-\|P_{k_{j+1}}(P_{k_{j'}}(I_\nu(f_{j'})))\| \\
	& & \text{(because $P_{k_{j+1}}\circ P_{k_{j'}}=P_{k_{j+1}}$)} \\
	& \geq & \rho^{-1}-\|Q_{k_{j+1}}(I_\nu(f_j))\|-\|P_{k_{j'}}(I_\nu(f_{j'}))\| \\
	& & \text{(by \eqref{eqn:seminormalized} and $\|P_{k_{j+1}}\|=1$)} \\
	& \geq & \rho^{-1}-\frac{1}{2^{j+1}} -\frac{1}{2^{j'+1}}  \\
	& & \text{(by ($\alpha$) and ($\beta$) with $g=\chi_\Omega$)} \\
	&>& 
	\rho^{-1}-\frac{1}{2^{j_0}} \geq \epsilon.
\end{eqnarray*}
The proof is finished.
\end{proof}

\subsection{Asplund spaces are $\mathcal{A}_{cc}$-variation admissible}

Let $(X_n)_{n\in \N}$ be a Schauder decomposition of a Banach space~$X$.
By a {\em block sequence with respect to $(X_n)_{n\in \N}$} we mean 
a sequence $(x_j)_{j\in \N}$ in~$X$ for which there is a sequence $(I_j)_{j\in \N}$ of non-empty finite subsets of~$\N$ such that 
$\max(I_j)<\min(I_{j+1})$ and $x_j\in \bigoplus_{n\in I_j}X_n$ for all $j\in \N$. 
We say that $(X_n)_{n\in \N}$ is {\em shrinking} if $\|P_k^*(x^*)-x^*\|\to 0$ as $k\to \infty$ for every $x^*\in X^*$,
where $P_k$ is the associated projection from $X$ onto $\bigoplus_{n=1}^k X_n$.
When $X$ has a Schauder basis $(e_n)_{n\in \N}$ and each $X_n$ is the subspace generated by~$e_n$, then 
$(X_n)_{n\in \N}$ is shrinking if and only if $(e_n)_{n\in \N}$ is shrinking in the usual sense.

The following fact belongs to the folklore and can be proved as in the case of Schauder bases
(see, e.g., \cite[Proposition~3.2.7]{alb-kal}).

\begin{pro}\label{pro:shrinking}
Let $(X_n)_{n\in \N}$ be a Schauder decomposition of a Banach space~$X$. The following statements are equivalent:
\begin{enumerate}
\item[(i)] $(X_n)_{n\in \N}$ is shrinking.
\item[(ii)] Every bounded block sequence with respect to $(X_n)_{n\in \N}$ is weakly null. 
\end{enumerate}
\end{pro}

\begin{theo}\label{theo:general}
Let $X$ be a Banach space having a shrinking Schauder decomposition $(X_n)_{n\in \N}$ such that $X_n$ is finite-dimensional for all $n\in \N$.
Then $X$ is $\mathcal{A}_{cc}$-variation admissible. In particular, every Banach space having a shrinking Schauder basis is $\mathcal{A}_{cc}$-variation admissible.
\end{theo}
\begin{proof}
Let $(\Omega,\Sigma)$ be a measurable space and let $\nu\in {\rm ca}(\Sigma,X)$ such that
$I_\nu$ is Dunford-Pettis. Then $\mathcal{R}(\nu)$ is relatively norm compact (Proposition~\ref{pro:charDP}). 
Fix $k\in \N$ and denote by $P_k$ the associated projection from~$X$ onto $\bigoplus_{n=1}^k X_n$. Since
$\bigoplus_{n=1}^k X_n$ is finite-dimensional, we have $|P_k \circ \nu|(\Omega)<\infty$.
By renorming, we can assume that $\|P_k\|=1$ for all $k\in \N$.

Suppose, by contradiction, that $|\nu|(\Omega)=\infty$. Let $(f_j)_{j\in \N}$ and $(R_j)_{j\in \N}$ be as in 
Lemma~\ref{lem:projection2}. Since $(I_\nu(f_j))_{j\in \N}$ is not norm convergent (by property~(ii) in Lemma~\ref{lem:projection2})
and $I_\nu$ is Dunford-Pettis, 
the sequence $(f_j)_{j\in \N}$ is not weakly convergent in~$L_1(\nu)$. 
In addition, $\|f_j\|_{L_1(\nu)}=1$ for all $j\in \N$. By Proposition~\ref{pro:convergenceproperty}, there is $g\in B_{L_\infty(\nu)}$ such that the sequence
$(I_\nu(f_j g))_{j\in \N}$ is not weakly null in~$X$. Then 
$(R_j(I_\nu(f_j g)))_{j\in \N}$ is a bounded block sequence with respect to $(X_n)_{n\in \N}$ which cannot be weakly null, by property~(i) in Lemma~\ref{lem:projection2}.
This contradicts that $(X_n)_{n\in \N}$ is shrinking (Proposition~\ref{pro:shrinking}).
\end{proof}

The last ingredient of our proof that Asplund spaces are $\mathcal{A}_{cc}$-variation admissible is the following deep result
of Zippin~\cite{zip} (cf. \cite[Theorem~III.1]{gho-alt2} and \cite{sch-7}):

\begin{theo}[Zippin]\label{theo:Zippin}
Every Banach space having separable dual is isomorphic to a subspace of a Banach space having a shrinking Schauder basis. 
\end{theo}

\begin{cor}\label{cor:Asplund}
Every Asplund space is $\mathcal{A}_{cc}$-variation admissible.
\end{cor}
\begin{proof} By Proposition~\ref{pro:stability2}(ii), it suffices to prove that every Banach space having separable dual is $\mathcal{A}_{cc}$-admissible.
 Since every Banach space having a shrinking Schauder basis is $\mathcal{A}_{cc}$-variation admissible
(Theorem~\ref{theo:general}), the conclusion follows from Theorem~\ref{theo:Zippin} and Proposition~\ref{pro:stability1}(ii). 
\end{proof}

\subsection{An application of the Davis-Figiel-Johnson-Pe{\l}czy\'{n}ski factorization}

We begin by recalling the refinement of the DFJP factorization developed by Lima, Nygaard and Oja in~\cite{lim-nyg-oja}. 
Let $Z$ and $X$ be Banach spaces, let $T:Z\to X$ be a (non-zero) operator and consider the set
$$
	K:=\frac{1}{\|T\|}\overline{T(B_Z)} \sub B_X.
$$ 
Fix $a \in (1,\infty)$ and write
$$
	f(a):=\left(\sum_{n=1}^\infty\frac{a^{n}}{(a^{n}+1)^2}\right)^{1/2}.
$$
For each $n\in \N$, let $\|\cdot\|_n$ be the Minkowski functional
of $K_n:=a^{n/2} K +a^{-n/2} B_X$, that is,
$$
	\|x\|_n:=\inf\{t>0: x\in tK_n\}
	\quad
	\text{for all $x\in X$.}
$$
The following theorem can be found in \cite[Lemmas 1.1 and 2.1, Theorem~2.2]{lim-nyg-oja}, 
with the exception of part (vi), which can be obtained similarly as for the usual DFJP factorization (see, e.g., \cite[\S3]{bou-J}).

\begin{theo}\label{theo:DFJP-LNO}
Under the previous assumptions, the following statements hold:
\begin{enumerate}
\item[(i)] $Y:=\{x\in X: \ \sum_{n=1}^\infty \|x\|_n^2 <\infty\}$
is a Banach space with the norm 
$$
	\|x\|_Y:=\left(\sum_{n=1}^\infty \|x\|_n^2\right)^{1/2}.
$$
\item[(ii)] $K \sub f(a) B_{Y}$ and the identity map $J: Y \to X$ is an operator.
\item[(iii)] $T$ factors as
\begin{equation}\label{eqn:DFJPLNO}
	\xymatrix@R=3pc@C=3pc{Z
	\ar[r]^{T} \ar[d]_{S} & X\\
	Y  \ar@{->}[ur]_{J}  & \\
	}
\end{equation}
where $S$ is an operator. 
\item[(iv)] $J$ is a norm-to-norm homeomorphism when restricted to~$K$. In fact:
$$
	\|x\|_Y^2 \leq \Big(\frac{1}{4}+\frac{1}{\ln a}\Big) \|x\| \quad \text{for all $x\in K$}.
$$
Therefore, if $T$ is Dunford-Pettis, then $S$ is Dunford-Pettis as well.
\item[(v)] If $T$ is weakly compact, then $Y$ is reflexive. 
\item[(vi)] If $T$ is Asplund, then $Y$ is Asplund. 
\item[(vii)] If $a$ is the unique element of~$(1,\infty)$ satisfying $f(a)=1$, then $\|S\|=\|T\|$ and $\|J\|=1$.
In this case, \eqref{eqn:DFJPLNO} is called the {\em DFJP-LNO factorization} of~$T$.
\end{enumerate}
\end{theo}

In \cite{nyg-rod} the DFJP-LNO factorization was applied to the integration operator of a vector measure.
Our next proposition gathers some of the results obtained in \cite[Theorems~3.7 and~4.5]{nyg-rod}:

\begin{pro}\label{theo:Im}
Let $X$ be a Banach space, let $(\Omega,\Sigma)$ be a measurable space, let $\nu\in {\rm ca}(\Sigma,X)$ and let 
$$
	\xymatrix@R=3pc@C=3pc{L_1(\nu)
	\ar[r]^{I_\nu} \ar[d]_{S} & X\\
	Y  \ar@{->}[ur]_{J}  & \\
	}
$$
be the DFJP-LNO factorization of~$I_\nu$. Define $\tilde{\nu}:\Sigma\to Y$ by $\tilde{\nu}(A):=S(\chi_A)$ for all $A\in \Sigma$. Then:
\begin{enumerate}
\item[(i)] $\tilde{\nu} \in {\rm ca}(\Sigma,Y)$, $\nu=J\circ \tilde{\nu}$ and $\mathcal{N}(\nu)=\mathcal{N}(\tilde{\nu})$.
\item[(ii)] $L_1(\tilde{\nu})=L_1(\nu)$, with $\|f\|_{L_1(\nu)}=\|f\|_{L_1(\tilde{\nu})}$ for all $f\in L_1(\nu)$, and $S=I_{\tilde{\nu}}$.
\item[(iii)] $\tilde{\nu}$ has finite (resp., $\sigma$-finite) variation whenever $\nu$ does.
\end{enumerate}
\end{pro}

\begin{cor}\label{cor:Asplundoperator}
Let $X$ be a Banach space, let $(\Omega,\Sigma)$ be a measurable space and let $\nu\in {\rm ca}(\Sigma,X)$.
If $I_\nu$ is Asplund and Dunford-Pettis, then $|\nu|(\Omega)<\infty$ and the inclusion map $\iota_\nu: L_1(|\nu|) \to L_1(\nu)$
is a lattice-isomorphism.
\end{cor}
\begin{proof}
Let $Y$, $J$ and $\tilde{\nu}$ be as in Proposition~\ref{theo:Im}. Since $I_{\tilde{\nu}}$ is Dunford-Pettis and $Y$ is Asplund
(Theorem~\ref{theo:DFJP-LNO}, parts (iv) and (vi)), we can apply Corollary~\ref{cor:Asplund} to get $|\tilde{\nu}|(\Omega)<\infty$,
hence $|\nu|(\Omega)=|J\circ \tilde{\nu}|(\Omega) \leq |\tilde{\nu}|(\Omega)<\infty$. 
This shows that every Banach space is $\mathcal{A}$-variation admissible, where 
$\mathcal{A}$ denotes the operator ideal of all Asplund and Dunford-Pettis operators.
The last statement follows from Proposition~\ref{pro:OIP}. 
\end{proof}

\section{Dunford-Pettis type properties}\label{section:DP}

\subsection{A remark on equimeasurability}

Let $(\Omega,\Sigma,\mu)$ be a finite measure space. A set $H \sub L_\infty(\mu)$ is said to be {\em equimeasurable} if for every
$\epsilon>0$ there is $A\in \Sigma$ with $\mu(\Omega \setminus A)\leq \epsilon$ such that $\{h\chi_A:h\in H\}$ is relatively norm compact
in~$L_\infty(\mu)$. Theorem~\ref{theo:equimeasurable} below is a particular case of \cite[Theorem~5.5.4]{bou-J}. We include
a direct proof for the sake of completeness. 

\begin{theo}\label{theo:equimeasurable}
Let $(\Omega,\Sigma,\mu)$ be a finite measure space. If $H \sub L_\infty(\mu)$ is relatively weakly compact, then it is equimeasurable.
\end{theo}
\begin{proof}
By the Davis-Figiel-Johnson-Pe{\l}czy\'{n}ski factorization (see, e.g., \cite[Theorem~5.37]{ali-bur}), there exist a reflexive Banach space~$Y$ and an operator $T:Y \to L_\infty(\mu)$
such that $T(B_Y) \supseteq H$. Let $i: L_1(\mu) \to L_\infty(\mu)^*$ be the inclusion operator and let $S:=T^*\circ i: L_1(\mu) \to Y^*$.
Since $Y^*$ is reflexive, $S$ is representable, that is, there is $g\in L_\infty(\mu,Y^*)$ such that
$$
	S(f)=\text{(Bochner)-}\int_\Omega fg \, d\mu \quad\text{for all $f\in L_1(\mu)$}
$$
(see, e.g., \cite[p.~75, Theorem~12]{die-uhl-J}).

Fix $\epsilon>0$. Since $g$ is strongly $\mu$-measurable, Egorov's theorem ensures the existence of $A\in \Sigma$ with $\mu(\Omega \setminus A)\leq \epsilon$
and a sequence $g_n:\Omega \to Y^*$ of $\Sigma$-simple $Y^*$-valued functions such that 
\begin{equation}\label{eqn:simple}
	\|g(t)-g_n(t)\| \leq \frac{1}{n} \quad\text{for every $t\in A$ and for every $n\in \N$}.
\end{equation}
For each $n\in \N$, let us consider the operator $S_n:L_1(\mu)\to Y^*$ defined by 
$$
	S_n(f)=\text{(Bochner)-}\int_A fg_n \, d\mu \quad\text{for all $f\in L_1(\mu)$}.
$$
Note that $S_n$ is a finite-rank operator, because $g_n$ is the sum of finitely many functions of the form $y^*\chi_B$, where $y^*\in Y^*$ and $B\in \Sigma$. 
Hence, $S_n$ is compact. Moreover, if $P_A:L_1(\mu)\to L_1(\mu)$ is 
the projection defined by $P_A(f):=f\chi_A$ for all $f\in L_1(\mu)$, then the operator $S\circ P_A: L_1(\mu)\to Y^*$ satisfies 
$$
	\|S\circ P_A-S_n\|=\sup_{f\in B_{L_1(\mu)}}\left\|\text{(Bochner)-}\int_A f (g-g_n) \, d\mu \right\| \stackrel{\eqref{eqn:simple}}{\leq}
	\frac{1}{n}. 
$$
It follows that $(S_n)_{n\in \N}$ converges to $S\circ P_A$ in the operator norm. In particular, $S\circ P_A$ is compact and, therefore, $(S\circ P_A)^*: Y \to L_\infty(\mu)$
is compact as well (by Schauder's theorem). 

For every $y\in Y$ and for every $f\in L_1(\mu)$ we have
\begin{multline*}
	\langle (S\circ P_A)^*(y),f \rangle=
	\langle y,(S\circ P_A)(f) \rangle=
	\langle y,T^*(i(f\chi_A)) \rangle \\ =
	\langle T(y),f\chi_A \rangle = \int_AfT(y)\, dy=\langle T(y)\chi_A,f \rangle.
\end{multline*} 
Therefore $(S\circ P_A)^*(y)=T(y)\chi_A$ for all $y \in Y$.
It follows that 
$$
	\{h\chi_A: \, h\in H\} \sub \{T(y)\chi_A: \, y\in B_Y\}=
	(S\circ P_A)^*(B_Y)
$$
and so $\{h\chi_A: \, h\in H\}$ is relatively norm compact in~$L_\infty(\mu)$.
\end{proof}

\subsection{A Dunford-Pettis type property for $L_1$ of a vector measure}

Recall that a Banach space $Z$ has the Dunford-Pettis property if and only if
$z_n^*(z_n)\to 0$ as $n\to \infty$ for all weakly null sequences $(z_n)_{n\in\N}$ and $(z_n^*)_{n\in\N}$ in~$Z$
and~$Z^*$, respectively (see, e.g., \cite[Theorem~5.4.4]{alb-kal}).

We next show that the $L_1$ space of an arbitrary vector measure enjoys
a Dunford-Pettis type property with respect to the ``vector duality'' 
induced by the integration operator (Subsection~\ref{subsection:vectorduality}).

\begin{theo}\label{theo:DPintegrationoperator}
Let $X$ be a Banach space, let $(\Omega,\Sigma)$ be a measurable space and let $\nu\in {\rm ca}(\Sigma,X)$.
Let $(f_n)_{n\in \N}$ be a sequence in~$L_1(\nu)$ and let $(g_n)_{n\in \N}$ be a weakly null sequence in~$L_\infty(\nu)$.
\begin{enumerate}
\item[(i)] If $(f_n)_{n\in \N}$ is weakly null, then $(I_\nu(f_ng_n))_{n\in \N}$ is weakly null.
\item[(ii)] If $(f_n)_{n\in \N}$ is bounded and equi-integrable, then $(I_\nu(f_ng_n))_{n\in \N}$ is norm null.
\end{enumerate}
\end{theo}
 \begin{proof} (i) Fix $x^*\in X^*$. Let $h_{x^*} \in L_\infty(|x^* \nu|)$ be the Radon-Nikod\'{y}m derivative of 
 $x^*\nu$ with respect to~$|x^*\nu|$. For each $n\in \N$
 we have
 \begin{equation}\label{eqn:weak}
 	x^*\big(I_\nu(f_n g_n)\big)=\int_\Omega f_n g_n \, d(x^*\nu)=\int_\Omega f_n h_{x^*} g_n \, d|x^*\nu|.
 \end{equation}
 Since $(f_n)_{n\in \N}$ is weakly null in~$L_1(\nu)$ and the inclusion map $L_1(\nu)\to L_1(|x^*\nu|)$ is an operator, 
$(f_n)_{n\in \N}$ is weakly null in $L_1(|x^*\nu|)$ and so the same holds
 for $(f_n h_{x^*})_{n\in \N}$. In the same way, $(g_n)_{n\in \N}$ is weakly null in $L_\infty(|x^* \nu|)$, so we can apply the Dunford-Pettis property
 of $L_1(|x^*\nu|)$ and~\eqref{eqn:weak} to conclude that $x^*(I_\nu(f_n g_n))\to 0$ as $n\to \infty$.
 Since $x^*\in X^*$ is arbitrary, $(I_\nu(f_ng_n))_{n\in \N}$ is weakly null.
 
(ii) Define $\alpha:=\sup_{n\in \N}\|f_n\|_{L_1(\nu)}$ and $\beta:=\sup_{n\in \N}\|g_n\|_{L_\infty(\nu)}$. 
Let $\mu$ be a Rybakov control measure of~$\nu$. Fix $\epsilon>0$.
Since $(f_n)_{n\in \N}$ is equi-integrable, we can choose $\delta>0$ such that 
\begin{equation}\label{eqn:UI}
	\sup_{f\in F}\|f\chi_B\|_{L_1(\nu)}\leq \epsilon
	\quad \text{for every $B\in \Sigma$ with $\mu(B)\leq \delta$}.
\end{equation}
By Theorem~\ref{theo:equimeasurable}, the set $\{g_n:n\in \N\}$ is equimeasurable, so
there is $A\in \Sigma$ with $\mu(\Omega \setminus A)\leq \delta$ such that $\{g_n \chi_A: n\in \N\}$ is relatively norm compact in~$L_\infty(\nu)$.
Since the sequence $(g_n \chi_A)_{n\in \N}$ is weakly null in $L_\infty(\nu)$
(bear in mind that the map $g \mapsto g\chi_A$ is an operator on~$L_\infty(\nu)$), we conclude that $(g_n \chi_A)_{n\in \N}$ is norm null in~$L_\infty(\nu)$. 
Choose $n_0\in \N$ such that 
\begin{equation}\label{eqn:norm-null}
	\sup_{n\geq n_0} \|g_n\chi_A\|_{L_\infty(\nu)} \leq \epsilon. 
\end{equation}
Now, for every $f\in F$ and for every $n\in \N$ with $n\geq n_0$ we have
\begin{eqnarray*}
	\big\|I_\nu(fg_n)\big\| &\leq& \big\|I_\nu(fg_n\chi_{\Omega \setminus A})\big\|+\big\|I_\nu(f g_n\chi_A)\big\|
	\\ &\stackrel{\text{(Prop.~\ref{pro:norming}(i))}}{\leq}&
	\|f\chi_{\Omega\setminus A}\|_{L_1(\nu)} \|g_n\|_{L_\infty(\nu)}+\|f\|_{L_1(\nu)} \|g_n\chi_A\|_{L_\infty(\nu)} 
	\\ &\stackrel{\text{\eqref{eqn:UI} \& \eqref{eqn:norm-null}}}{\leq} & (\beta+\alpha)\epsilon.
\end{eqnarray*}
As $\epsilon>0$ is arbitrary, the sequence $(I_\nu(f_ng_n))_{n\in \N}$ is norm null.
\end{proof}

\subsection{The positive Schur property as a Dunford-Pettis type property}

As a natural outcome of our previous work we get the following characterization:

\begin{theo}\label{theo:PSP}
Let $X$ be a Banach space, let $(\Omega,\Sigma)$ be a measurable space and let $\nu\in {\rm ca}(\Sigma,X)$.
The following statements are equivalent:
\begin{enumerate}
\item[(i)] $L_1(\nu)$ has the positive Schur property.
\item[(ii)] For all weakly null sequences $(f_n)_{n\in \N}$ and $(g_n)_{n\in \N}$ in $L_1(\nu)$ and $L_\infty(\nu)$, respectively, the sequence
$(I_\nu(f_ng_n))_{n\in \N}$ is norm null.
\end{enumerate}
 \end{theo}
 \begin{proof}
 (i)$\impli$(ii): This follows from Theorem~\ref{theo:DPintegrationoperator}, because the positive Schur property of~$L_1(\nu)$ is equivalent 
 to the fact that every relatively weakly compact subset of $L_1(\nu)$ is equi-integrable (Propositions~\ref{pro:PSPabstract} and~\ref{pro:Lweaklycompact}).  
 
(ii)$\impli$(i): By Propositions~\ref{pro:PSPabstract} and~\ref{pro:Lweaklycompact}, it suffices to prove that 
\emph{every disjoint weakly null sequence $(f_n)_{n\in \N}$ in~$L_1(\nu)$ is equi-integrable}. 
Let $(A_n)_{n\in \N}$ be a sequence of pairwise disjoint elements of~$\Sigma$ such that $f_n\chi_{A_n}=f_n$ for all $n\in \N$. Observe 
that $(\chi_{A_n})_{n\in \N}$ is weakly null in~$L_\infty(\nu)$.
Indeed, we can assume without loss of generality that $A_n\not\in\mathcal{N}(\nu)$ for all $n\in \N$. Then 
$(\chi_{A_n})_{n\in \N}$ is a basic sequence in~$L_\infty(\nu)$ which is equivalent to the usual basis of~$c_0$. In particular,
$(\chi_{A_n})_{n\in \N}$ is weakly null in~$L_\infty(\nu)$.

Fix $A\in \Sigma$. Define $\tilde{f}_n:=f_n\chi_A$ for all $n\in \N$. Note that 
\begin{equation}\label{eqn:star}
	I_\nu(\tilde{f}_n\chi_{A_n})=I_\nu(\tilde{f}_n)=\nu_{f_n}(A)
	\quad\text{for all $n\in \N$}.
\end{equation}
Since $(\tilde{f}_n)_{n\in \N}$ is weakly null in~$L_1(\nu)$ (because $(f_n)_{n\in \N}$ is weakly null and
the map $h \mapsto h\chi_A$
is an operator on~$L_1(\nu)$) and $(\chi_{A_n})_{n\in \N}$ is weakly null in~$L_\infty(\nu)$, 
condition~(ii) and~\eqref{eqn:star} imply that the sequence $(\nu_{f_n}(A))_{n\in \N}$ 
is norm null. As $A\in \Sigma$ is arbitrary, we can apply Lemma~\ref{lem:equi} to conclude that $(f_n)_{n\in \N}$ is equi-integrable.
\end{proof}

Of course, Theorems~\ref{theo:DPintegrationoperator} and~\ref{theo:PSP} provide another point of view for
the positive Schur property of the $L_1$ space of a vector measure taking values in a Banach space with the Schur property, \cite[proof of Theorem~4]{cur3}. 

\subsection{Vector measures with $\sigma$-finite variation}
 
The analysis of the Dunford-Pettis property is simpler for $L_1$ spaces of a vector measure with $\sigma$-finite variation. 
 
\begin{pro}\label{pro:finite-variation}
Let $X$ be a Banach space, let $(\Omega,\Sigma)$ be a measurable space and let $\nu\in {\rm ca}(\Sigma,X)$
with $\sigma$-finite variation. If $(f_n)_{n\in \N}$ is a bounded and equi-integrable sequence in~$L_1(\nu)$ and 
$(\varphi_n)_{n\in \N}$ is a weakly null sequence in~$L_1(\nu)^*$, then $\varphi_n(f_n)\to 0$ as $n\to\infty$.
\end{pro}
\begin{proof} The sequence $(f_n)_{n\in \N}$ is approximately order bounded (Proposition~\ref{pro:Lweaklycompact}). Hence,
we can assume without loss of generality that $f_n\in j_\nu(B_{L_\infty(\nu)})$ for all $n\in \N$.
Define $\alpha:=\sup_{n\in \N}\|\varphi_n\|_{L_1(\nu)^*}$.
Let $(A_m)_{m\in \N}$ be an increasing sequence in~$\Sigma$ such that $\Omega=\bigcup_{m\in \N}A_m$ and
$|\nu|(A_m)<\infty$ for all $m\in \N$. 
Fix $\epsilon>0$. Choose $m\in \N$ large enough such that
\begin{equation}\label{eqn:smalltail}
	\|\nu\|(\Omega \setminus A_m)\leq \epsilon.
\end{equation}
Define $\mu(A):=|\nu|(A\cap A_m)$ for all $A\in \Sigma$, so that $\mu$ is a finite non-negative measure. Consider  
the inclusion operator $\iota: L_1(\mu) \to L_1(\nu)$ (see, e.g., \cite[Lemma~3.14]{oka-alt})
and $\iota^*:L_1(\nu)^*\to L_\infty(\mu)$.
Define $g_n:=\iota^*(\varphi_n)\in L_\infty(\mu)$ for all $n\in \N$, so that
$(g_n)_{n\in\N}$ is weakly null in $L_\infty(\mu)$.

The sequence $(f_n\chi_{A_m})_{n\in \N}$ is bounded and equi-integrable in $L_1(\mu)$ and 
$$
	\langle g_n,f_n \chi_{A_m}\rangle = \int_{A_m} f_n g_n \, d\mu=\varphi_n(f_n \chi_{A_m}) \quad\text{for all $n\in \N$}.
$$
Therefore, the Dunford-Pettis property of~$L_1(\mu)$ (cf. Theorem~\ref{theo:DPintegrationoperator}(ii)) 
ensures that $\varphi_n(f_n \chi_{A_m})\to 0$ as $n\to \infty$.
Take $n_0\in \N$ such that 
$$
	|\varphi_n(f_n \chi_{A_m})|\leq \epsilon \quad \text{whenever $n\geq n_0$}.
$$
Since
$$
	|\varphi_n(f_n\chi_{\Omega \setminus A_m})| \leq 
	\alpha \|f_n\chi_{\Omega \setminus A_m}\|_{L_1(\nu)} \leq \alpha\|\nu\|(\Omega\setminus A_m)\leq \alpha\epsilon
	\quad\text{for all $n\in \N$}
$$ 
(by Proposition~\ref{pro:norming}(i) and~\eqref{eqn:smalltail}), 
we have
$$
	| \varphi_n(f_n)|\leq |\varphi_n(f_n\chi_{A_m})| + |\varphi_n(f_n\chi_{\Omega \setminus A_m})|
	\leq (1+\alpha)\epsilon
	\quad\text{whenever $n\geq n_0$}.
$$ 
This shows that $\varphi_n(f_n)\to 0$ as $n\to\infty$.
\end{proof}

By putting together Propositions~\ref{pro:PSPabstract}, \ref{pro:Lweaklycompact} and~\ref{pro:finite-variation}, we get
the already mentioned result from~\cite{cur3}: 

\begin{cor}\label{cor:PSP-DP}
Let $X$ be a Banach space, let $(\Omega,\Sigma)$ be a measurable space and let $\nu\in {\rm ca}(\Sigma,X)$
with $\sigma$-finite variation. If $L_1(\nu)$ has the positive Schur property, then it has the Dunford-Pettis property.
\end{cor}

Let $E$ be a Banach space with a normalized $1$-unconditional Schauder basis, say~$(e_n)_{n\in \N}$.
The $E$-sum of countably many copies of~$L_1[0,1]$ is the Banach lattice
$Z$ of all sequences $(h_n)_{n\in \N}$ in~$L_1[0,1]$ such that the series $\sum_{n=1}^\infty \|h_n\|_{L_1[0,1]} \, e_n$
converges unconditionally in~$E$, equipped with the norm
$$
	\left\|(h_n)_{n\in \N}\right\|_{Z}
	:=
	\left\|\sum_{n=1}^\infty \|h_n\|_{L_1[0,1]} \, e_n\right\|_E
$$
and the coordinatewise order. If $E$ has the the Schur property, then $Z$ has the positive Schur property, but it is 
not lattice-isomorphic to an AL-space unless $E$ is isomorphic to~$\ell_1$, \cite[Section~3]{wnu3}.

The following construction provides more examples of Banach lattices with such features:
 
\begin{exa}\label{exa:Lipecki}
Let $X$ be a Banach space and let $\sum_{n=1}^\infty x_n$ be an unconditionally convergent series in~$X$ with $x_n\neq 0$ for all $n\in \N$. Let 
$\lambda$ be the Lebesgue measure on the $\sigma$-algebra $\Sigma$ of all Borel subsets of~$[0,1]$.
Write $I_n:=(2^{-n},2^{-n+1}]$ for all $n\in \N$. Then:
\begin{enumerate}
\item[(i)] The formula
$$
	\nu(A):=\sum_{n=1}^\infty 2^n\lambda (A\cap I_n)x_n, \quad A\in \Sigma,
$$ 
defines a vector measure $\nu \in {\rm ca}(\Sigma,X)$.
\item[(ii)] $\mathcal{N}(\nu)=\mathcal{N}(\lambda)$. Hence, $\nu$ is atomless and $L_1(\nu)$ is separable.
\item[(iii)] $\mathcal{R}(\nu)$ is relatively norm compact.
\item[(iv)] $|\nu|$ is $\sigma$-finite and $|\nu|([0,1])=\sum_{n=1}^\infty \|x_n\|$.
\item[(v)] If $\sum_{n=1}^\infty x_n$ is not absolutely convergent, then $L_1(\nu)$ is not lattice-isomorphic to an AL-space.
\item[(vi)] If $X$ has the Schur property, then $L_1(\nu)$ has the positive Schur property and the Dunford-Pettis property.
\item[(vii)] If $\sum_{n=1}^\infty x_n$ is not absolutely convergent and $X$ has the Schur property, then 
$L_1(\nu)$ is not lattice-isomorphic to $L_1(\tilde{\nu})$ for any $\sigma$-algebra~$\tilde{\Sigma}$
and any $\tilde{\nu} \in {\rm ca}(\tilde{\Sigma},c_0)$ such that $\mathcal{R}(\tilde{\nu})$ is relatively norm compact. 
\end{enumerate}
\end{exa}
\begin{proof}
Since $\sum_{n=1}^\infty x_n$ is unconditionally convergent, for every $(a_n)_{n\in \N}\in \ell_\infty$
the series $\sum_{n=1}^\infty a_nx_n$ is unconditionally convergent and the map
$$
	T:\ell_\infty \to X, \quad
	T((a_n)_{n\in \N}):=\sum_{n=1}^\infty a_n x_n,
$$
is a compact operator (see, e.g., \cite[Theorem~1.9]{die-alt}). This shows that the map $\nu$ is well-defined and has
relatively norm compact range (note that $2^n \lambda(A\cap I_n) \leq 1$ for all $n\in \N$). 
Since the map $\Sigma\ni A\mapsto 2^n\lambda(A\cap I_n)x_n \in X$ is countably additive for each $n\in \N$,
the Vitali-Hahn-Saks theorem (see, e.g., \cite[p.~24, Corollary~10]{die-uhl-J}) ensures that
$\nu$ is countably additive. This proves parts (i) and~(iii).

(ii) The equality $\mathcal{N}(\nu)=\mathcal{N}(\lambda)$ is obvious. Since $\lambda$ is atomless, so is~$\nu$.
Let $\mathcal{C} \sub \Sigma$ be a countable set such that for every $A\in \Sigma$ we have
$\inf_{C\in \mathcal{C}}\lambda(A\triangle C)=0$. Then for every $A\in \Sigma$ we also have $\inf_{C\in \mathcal{C}}\|\nu\|(A\triangle C)=0$
(notice that $\nu$ is $\lambda$-continuous). 
This implies that~$L_1(\nu)$ is separable, because
the set of all $\Sigma$-simple functions is norm dense in~$L_1(\nu)$.

(iv) It is easy to check that $|\nu|(A)=\sum_{n=1}^\infty 2^n\lambda (A\cap I_n)\|x_n\|$ for every $A\in \Sigma$. 

(v) This follows from \cite[Proposition~2]{cur2} and~(iv).
 
(vi) We already know that the Schur property of~$X$ implies that $L_1(\nu)$ has the positive Schur property, \cite[proof of Theorem~4]{cur3}. 
Now, (iv) and Corollary~\ref{cor:PSP-DP} ensure that $L_1(\nu)$ has the Dunford-Pettis property.

(vii) Suppose, by contradiction, that there exist a $\sigma$-algebra $\tilde{\Sigma}$ and $\tilde{\nu}\in {\rm ca}(\tilde{\Sigma},c_0)$ 
such that $\mathcal{R}(\tilde{\nu})$ is relatively norm compact and $L_1(\nu)$ is lattice-isomorphic to $L_1(\tilde{\nu})$. 
Then $L_1(\tilde{\nu})$ has the positive Schur property (by~(vi)) and we can apply Proposition~\ref{pro:charDP}
to infer that the integration operator $I_{\tilde{\nu}}:L_1(\tilde{\nu}) \to c_0$ is Dunford-Pettis.
Now, Proposition~\ref{pro:OIP} and Theorem~\ref{theo:general} (the usual basis of $c_0$ is shrinking) imply that  
$L_1(\tilde{\nu})$ is lattice-isomorphic to an AL-space, which contradicts~(v).
\end{proof}

\begin{rem}\label{rem:Lipecki}
{\rm Part (vii) of Example~\ref{exa:Lipecki} should be compared with \cite[Theorem~1]{cur2}. That result states
that if $X$ is a Banach space, $(\Omega,\Sigma)$ is a measurable space, the vector measure $\nu\in {\rm ca}(\Sigma,X)$ is atomless and $L_1(\nu)$ is separable, then
there is $\tilde{\nu}\in {\rm ca}(\Sigma,c_0)$ such that $L_1(\nu)$ and $L_1(\tilde{\nu})$ are lattice-isometric
(cf. \cite[Theorem~5]{lip} for another proof). For variants in the non-separable setting, see \cite{rod16} and~\cite{rod22}.
In \cite[Theorem~5]{lip} it was claimed that if, in addition, $\mathcal{R}(\nu)$ is relatively norm compact, then $\tilde{\nu}$ can be chosen 
so that $\mathcal{R}(\tilde{\nu})$ is relatively norm compact as well.
Unfortunately, this turns out to be false in general, as shown in Example~\ref{exa:Lipecki}(vii).}
\end{rem}

\subsection*{Acknowledgements} 
I would like to thank Z.~Lipecki for valuable correspondence about Remark~\ref{rem:Lipecki}.
The research was supported by grants PID2021-122126NB-C32 
(funded by MCIN/AEI/10.13039/501100011033 and ``ERDF A way of making Europe'', EU) and 
21955/PI/22 (funded by {\em Fundaci\'on S\'eneca - ACyT Regi\'{o}n de Murcia}).


\begin{thebibliography}{10}

\bibitem{abr-woj}
Yu.~A. Abramovich and P.~Wojtaszczyk, \emph{The uniqueness of order in the spaces {$L\sb{p}[0,1]$} and {$l\sb{p}$}}, Mat. Zametki
  \textbf{18} (1975), no.~3, 313--325. 

\bibitem{alb-kal}
F.~Albiac and N.~J. Kalton, \emph{Topics in {B}anach space theory}, Graduate
  Texts in Mathematics, vol. 233, Springer, New York, 2006. 

\bibitem{ali-bur}
C.~D. Aliprantis and O.~Burkinshaw, \emph{Positive operators}, Springer,
  Dordrecht, 2006. 

\bibitem{bot-alt}
G.~Botelho, Q.~Bu, D.~Ji, and K.~Navoyan, \emph{The positive {S}chur property
  on positive projective tensor products and spaces of regular multilinear
  operators}, Monatsh. Math. \textbf{197} (2022), no.~4, 565--578. 

\bibitem{bou-J}
R.~D. Bourgin, \emph{Geometric aspects of convex sets with the
  {R}adon-{N}ikod\'ym property}, Lecture Notes in Mathematics, vol. 993,
  Springer-Verlag, Berlin, 1983. 

\bibitem{cal-alt-6}
J.~M. Calabuig, S.~Lajara, J.~Rodr{\'{\i}}guez, and E.~A.
  S{\'a}nchez-P{\'e}rez, \emph{Compactness in {$L^1$} of a vector measure},
  Studia Math. \textbf{225} (2014), no.~3, 259--282. 

\bibitem{cal-alt-5}
J.~M. Calabuig, J.~Rodr{\'{\i}}guez, and E.~A. S{\'a}nchez-P{\'e}rez, \emph{On
  completely continuous integration operators of a vector measure}, J. Convex
  Anal. \textbf{21} (2014), no.~3, 811--818. 

\bibitem{cal-alt-7}
J.~M. Calabuig, J.~Rodr\'{\i}guez, and E.~A. S\'{a}nchez-P\'{e}rez,
  \emph{Summability in {$L^1$} of a vector measure}, Math. Nachr. \textbf{290}
  (2017), no.~4, 507--519. 

\bibitem{che}
D.~Chen, \emph{Quantitative positive {S}chur property in {B}anach lattices},
  Proc. Amer. Math. Soc. \textbf{151} (2023), no.~3, 1167--1178. 

\bibitem{cur0}
G. P. Curbera, \emph{The space of integrable functions with respect to a vector
  measure}, Ph.D. Thesis, Universidad de Sevilla, 1992,
  https://hdl.handle.net/11441/76519.

\bibitem{cur1}
G.~P. Curbera, \emph{Operators into {$L\sp 1$} of a vector measure and
  applications to {B}anach lattices}, Math. Ann. \textbf{293} (1992), no.~2,
  317--330. 

\bibitem{cur2}
G.~P. Curbera, \emph{When {$L\sp 1$} of a vector measure is an {AL}-space}, Pacific
  J. Math. \textbf{162} (1994), no.~2, 287--303. 

\bibitem{cur3}
G.~P. Curbera, \emph{Banach space properties of {$L\sp 1$} of a vector measure},
  Proc. Amer. Math. Soc. \textbf{123} (1995), no.~12, 3797--3806. 

\bibitem{cur-ric-6}
G.~P. Curbera and W.~J. Ricker, \emph{On the {R}adon-{N}ikodym property in
  function spaces}, Proc. Amer. Math. Soc. \textbf{145} (2017), no.~2,
  617--626. 

\bibitem{cur-ric-5}
G.~P. Curbera and W.~J. Ricker, \emph{The weak {B}anach-{S}aks property for function spaces}, Rev. R.
  Acad. Cienc. Exactas F\'{\i}s. Nat. Ser. A Mat. RACSAM \textbf{111} (2017),
  no.~3, 657--671. 

\bibitem{hev-alt}
D.~de~Hevia, G.~Mart\'{i}nez-Cervantes, A.~Salguero-Alarc\'{o}n, and
  P.~Tradacete, \emph{A counterexample to the complemented subspace problem in
  {B}anach lattices}, arXiv:2310.02196.

\bibitem{die-alt}
J.~Diestel, H.~Jarchow, and A.~Tonge, \emph{Absolutely summing operators},
  Cambridge Studies in Advanced Mathematics, vol.~43, Cambridge University
  Press, Cambridge, 1995. 

\bibitem{die-uhl-J}
J.~Diestel and J.~J. Uhl, Jr., \emph{Vector measures}, Mathematical Surveys, No. 15, American Mathematical
  Society, Providence, R.I., 1977.

\bibitem{dep-alt}
P.~G. Dodds, B.~de~Pagter, and W.~Ricker, \emph{Reflexivity and order
  properties of scalar-type spectral operators in locally convex spaces},
  Trans. Amer. Math. Soc. \textbf{293} (1986), no.~1, 355--380. 

\bibitem{fab-ultimo}
M.~Fabian, P.~Habala, P.~H{\'a}jek, V.~Montesinos, and V.~Zizler, \emph{Banach
  space theory. The basis for linear and nonlinear analysis}, CMS Books in Mathematics, Springer, New York, 2011. 

\bibitem{gho-alt2}
N.~Ghoussoub, B.~Maurey, and W.~Schachermayer, \emph{Slicings, selections and
  their applications}, Canad. J. Math. \textbf{44} (1992), no.~3, 483--504.

\bibitem{les-alt}
K.~Le\'{s}nik, L.~Maligranda, and J.~Tomaszewski, \emph{Weakly compact sets and
  weakly compact pointwise multipliers in {B}anach function lattices}, Math.
  Nachr. \textbf{295} (2022), no.~3, 574--592. 

\bibitem{lim-nyg-oja}
A.~Lima, O.~Nygaard, and E.~Oja, \emph{Isometric factorization of weakly
  compact operators and the approximation property}, Israel J. Math.
  \textbf{119} (2000), 325--348. 

\bibitem{lip}
Z.~Lipecki, \emph{Semivariations of a vector measure}, Acta Sci. Math. (Szeged)
  \textbf{76} (2010), no.~3-4, 411--425. 

\bibitem{man-J}
G.~Manjabacas, \emph{Topologies associated to norming sets in {B}anach spaces},
  Ph.D. Thesis (Spanish), Universidad de Murcia, 1998,
  http://hdl.handle.net/10201/33837.

\bibitem{mey2}
P.~Meyer-Nieberg, \emph{Banach lattices}, Universitext, Springer-Verlag,
  Berlin, 1991. 

\bibitem{nyg-pol2}
O.~Nygaard and M.~P\~{o}ldvere, \emph{Families of vector measures of uniformly
  bounded variation}, Arch. Math. (Basel) \textbf{88} (2007), no.~1, 57--61.

\bibitem{nyg-rod}
O.~Nygaard and J.~Rodr\'{\i}guez, \emph{Isometric factorization of vector
  measures and applications to spaces of integrable functions}, J. Math. Anal.
  Appl. \textbf{508} (2022), no.~1, Paper No. 125857, 16 p. 

\bibitem{oka}
S.~Okada, \emph{The dual space of {$L\sp 1(\mu)$} for a vector measure
  {$\mu$}}, J. Math. Anal. Appl. \textbf{177} (1993), no.~2, 583--599.

\bibitem{oka-alt2}
S.~Okada, W.~J. Ricker, and L.~Rodr{\'{\i}}guez-Piazza, \emph{Compactness of
  the integration operator associated with a vector measure}, Studia Math.
  \textbf{150} (2002), no.~2, 133--149. 

\bibitem{oka-alt3}
S.~Okada, W.~J. Ricker, and L.~Rodr{\'{\i}}guez-Piazza, \emph{Operator ideal properties of vector measures with finite
  variation}, Studia Math. \textbf{205} (2011), no.~3, 215--249. 

\bibitem{oka-alt4}
S.~Okada, W.~J. Ricker, and L.~Rodr{\'{\i}}guez-Piazza, \emph{Operator ideal properties of the integration map of a vector
  measure}, Indag. Math. (N.S.) \textbf{25} (2014), no.~2, 315--340.

\bibitem{oka-alt}
S.~Okada, W.~J. Ricker, and E.~A. S{\'a}nchez~P{\'e}rez, \emph{Optimal domain
  and integral extension of operators. Acting in function
  spaces}, Operator Theory: Advances and
  Applications, vol. 180, Birkh\"auser Verlag, Basel, 2008. 

\bibitem{oka-rod-san}
S.~Okada, J.~Rodr\'{\i}guez, and E.~A. S\'{a}nchez-P\'{e}rez, \emph{On vector
  measures with values in {$\ell_\infty $}}, Studia Math. \textbf{274} (2024),
  no.~2, 173--199. 

\bibitem{rod15}
J.~Rodr\'{\i}guez, \emph{Factorization of vector measures and their integration
  operators}, Colloq. Math. \textbf{144} (2016), no.~1, 115--125. 

\bibitem{rod16}
J.~Rodr\'{\i}guez, \emph{On non-separable {$L^1$}-spaces of a vector measure}, Rev. R.
  Acad. Cienc. Exactas F\'{\i}s. Nat. Ser. A Mat. RACSAM \textbf{111} (2017),
  no.~4, 1039--1050. 

\bibitem{rod22}
J.~Rodr\'{\i}guez, \emph{On vector measures with values in $c_0(\kappa)$}, preprint.

\bibitem{rod-rue}
J.~Rodr\'{\i}guez and A.~Rueda~Zoca, \emph{On weakly almost square {B}anach
  spaces}, Proc. Edinb. Math. Soc. (2) \textbf{66} (2023), no.~4, 979--997.

\bibitem{san-hen}
J.~A. S\'{a}nchez~Henr\'{i}quez, \emph{Operadores en ret\'{i}culos de {B}anach:
  aplicaciones}, Ph.D. Thesis (Spanish), Universidad Complutense de Madrid,
  1985.

\bibitem{sch-7}
Th. Schlumprecht, \emph{On {Z}ippin's embedding theorem of {B}anach spaces into
  {B}anach spaces with bases}, Adv. Math. \textbf{274} (2015), 833--880.

\bibitem{wnu2}
W.~Wnuk, \emph{A note on the positive {S}chur property}, Glasgow Math. J.
  \textbf{31} (1989), no.~2, 169--172. 

\bibitem{wnu3}
W.~Wnuk, \emph{Some characterizations of {B}anach lattices with the {S}chur
  property}, Rev. Mat. Univ. Complutense Madr. \textbf{2} (1989), Suppl., 217--224. 

\bibitem{zip}
M.~Zippin, \emph{Banach spaces with separable duals}, Trans. Amer. Math. Soc.
  \textbf{310} (1988), no.~1, 371--379. 

\end{thebibliography}

\bibliographystyle{amsplain}

\end{document}